\documentclass{amsart}

\usepackage[utf8]{inputenc}
\usepackage[T1]{fontenc}
\usepackage[english]{babel}
\usepackage{amssymb}
\usepackage{mathrsfs,color,txfonts}
\theoremstyle{plain}
\newtheorem{theorem}{Theorem}
\newtheorem{lemma}[theorem]{Lemma}

\theoremstyle{remark}

\newtheorem*{remark}{Remark}

\DeclareMathOperator{\BMO}{BMO}
\DeclareMathOperator{\VMO}{VMO}
\DeclareMathOperator{\UC}{UC}
\DeclareMathOperator{\BV}{BV}
\DeclareMathOperator{\supp}{supp}
\DeclareMathOperator{\ba}{ba}
\DeclareMathOperator{\distanza}{dist}

\newcommand\fff{\mathcal{F}}
\newcommand\eee{\varepsilon}
\newcommand\ttt{\vartheta}
\newcommand\tvet{\boldsymbol{\ttt}}
\newcommand\vf{\varphi}

\newcommand{\reale}{{\mathbb R}}

\newcommand{\de}{\,{\rm d}}
\newcommand\bbb{\mathcal{B}}
\newcommand\ds{\displaystyle}

\begin{document}

\title[Atomic decompositions, two stars theorems, and distances]{Atomic decompositions, two stars theorems, and distances for the Bourgain--Brezis--Mironescu space and other big spaces}
\date{\today}

\author[L. D'Onofrio]{Luigi D'Onofrio}
\address{Dipartimento di Scienze e Tecnologie, Universit\`a degli Studi di Napoli ``Parthenope", Centro Direzionale Isola C4, 80100 Napoli, Italy}
\email{donofrio@uniparthenope.it}

\author[L. Greco]{Luigi Greco}
\address{Dipartimento di Ingegneria Elettrica e delle Tecnologie dell’Informazione, Universit\`a degli Studi di Napoli ``Federico II”, Via Claudio 21, 80125 Napoli, Italy}
\email{luigreco@unina.it}

\author[K.-M. Perfekt]{Karl-Mikael Perfekt}
\address{Department of Mathematics and Statistics,
	University of Reading, Reading RG6 6AX, United Kingdom}
\email{k.perfekt@reading.ac.uk}

\author[C. Sbordone]{Carlo Sbordone}
\address{Dipartimento di Matematica e Applicazioni ``R. Caccioppoli",
Universit\`a degli Studi di Napoli  ``Federico II",
Via Cintia, 80126 Napoli, Italy}
\email{sbordone@unina.it}

\author[R. Schiattarella]{Roberta Schiattarella}
\address{Dipartimento di Matematica e Applicazioni ``R. Caccioppoli",
Universit\`a degli Studi di Napoli  ``Federico II",
Via Cintia, 80126 Napoli, Italy}
\email{roberta.schiattarella@unina.it}


\thanks{ }

\begin{abstract}
Given a Banach space $E$ with a supremum-type norm induced by a collection of operators, we prove that $E$ is a dual space and provide an atomic decomposition of its predual. We apply this result, and some results obtained previously by one of the authors, to the function space $\bbb$ introduced recently by Bourgain, Brezis, and Mironescu. This yields an atomic decomposition of the predual $\bbb_\ast$, the biduality result that $\bbb_0^\ast = \bbb_\ast$ and $\bbb_\ast^\ast = \bbb$, and a formula for the distance from an element $f \in \bbb$ to $\bbb_0$.
\end{abstract}

\subjclass[2010]{}

\maketitle
\section{Introduction}
Suppose that a Banach space $E$ is defined and normed by the fact that $x \in E$ if and only if $\sup_{L \in \mathcal{L}} \|Lx\|_{Y} < \infty$, where $\mathcal{L}$ is a collection of operators $L \colon X \to Y$, $X$ and $Y$ Banach spaces. Spaces of this kind include  the space of bounded mean oscillation ($\BMO$), H\"older spaces, the space of bounded variation ($\BV$), Marcinkiewicz spaces, and various spaces of holomorphic functions of weighted or M\"obius invariant type.

Most of these spaces are known to have a predual $E_*$ whose elements can be defined in terms of a decomposition into designated ``atoms''. See for instance \cite{BB18, DHKY18} for two recent examples related to what will be our main application. While the general line of argument to obtain such atomic decompositions has appeared frequently and repeatedly, the main purpose of this note is to provide a completely functional analytic proof, independent of any particular structure of $\mathcal{L}$.

Another typical result, in the cases where it is possible to define a sufficiently rich ``vanishing'' subspace $E_0 \subset E$, is the ``two stars theorem''. Namely, that $E_0^\ast = E_\ast$. Furthermore, the distance from an element $f \in E$ to $E_0$ is usually given by an appropriate limit of the defining functionals. The pair $(\BMO, \VMO)$ provides a familiar example. These phenomena were formalized by one of the authors in \cite{Per13, Per17}, and they were proven to hold in a fairly general context (without giving any description of $E_\ast$). See also \cite{DSS18} for a survey.

Our second purpose is to demonstrate the application of these results to the new function space $\bbb$ introduced by Bourgain, Brezis, and Mironescu \cite{BBM15}. Recent work related to this space can also be found in \cite{ABBF, AC, FFGS, FMS1, FMS2}. For $d \geq 1$, $f \in L^1((0,1)^d)$, and a cube $Q_\varepsilon \subset (0,1)^d$ with sides of length $\varepsilon$ and parallel to the co-ordinate axes, let $f_{Q_\varepsilon}$ denote the average of $f$ on $Q_\varepsilon$. Define the norm (modulo constants)
\begin{equation}\label{BBMnorm}
\|f\|_{\bbb}:= \sup_{0< \varepsilon \leq 1} [f]_\varepsilon,
\end{equation}
where
$$[f]_\varepsilon=\varepsilon^{d-1} \sup_{\mathcal F_\varepsilon} \sum_{Q_{\varepsilon}\in \mathcal F_\varepsilon}  \frac{1}{|Q_{\varepsilon}|}\int_{Q_{\varepsilon}}\left|f(x)-f_{Q_\varepsilon}\right|\,dx\,.
$$
Here $\mathcal F_\epsilon$ denotes a collection of mutually disjoint $\varepsilon$-cubes $Q_\varepsilon \subset (0,1)^d$  such that the cardinality  $|\mathcal{F}_\varepsilon| \leq \varepsilon^{1-d}$, and the supremum is taken over all such collections. The space $\bbb$ is then defined as
$$\bbb = \{f \in L^1((0,1)^d) \, : \, \|f\|_{\bbb} < \infty\}.$$
When $d=1$, $\bbb = \BMO$. For $d\geq 2$, the $\bbb$-norm is strictly weaker than the $\BMO$-norm. In fact, both $\BMO$ and $\BV$ are continuously contained in $\bbb$ (see \cite{BBM15}).

The separable vanishing subspace $\bbb_0$ consists of those $f \in \bbb$ such that
$$\varlimsup_{\varepsilon \to 0} \;[f]_{\varepsilon}= 0.$$
$\VMO$ and $W^{1,1}$ are continuously contained in $\bbb_0$.

For the space $\bbb$, our result yields the following.
\begin{theorem} \label{thm:Bpredual}
$\bbb$ has an (isometric) predual $\bbb_\ast$. Every $\varphi \in \bbb_\ast$ is of the form
\begin{equation} \label{eq:Batdecomp}
\varphi = \sum_{n=1}^\infty \lambda_n g_n,
\end{equation}
where $(\lambda_n) \in \ell^1$ and each atom $g_n$ is associated with an $\varepsilon = \varepsilon_n$ and a collection of disjoint $\varepsilon$-cubes $\mathcal{F}_\varepsilon$ such that $|\mathcal{F}_\varepsilon| \leq \varepsilon^{1-d}$ and
	    \begin{itemize}
	\item $\supp g_n \subset \cup \mathcal{F}_{\varepsilon}$,
\item $|g_n|\chi_{Q_\varepsilon} \leq \varepsilon^{d-1} \frac{1}{|Q_\varepsilon|}$ for every $Q_\varepsilon \in \mathcal{F}_{\varepsilon}$,	
\item $\int_{Q_\varepsilon} g_n \, dx= 0$ for every $Q_\varepsilon \in \mathcal{F}_{\varepsilon}$.
	
\end{itemize}
The action of $f \in \bbb$ on $\varphi$ is given by
$$f(\varphi) = \sum_{n=1}^\infty \lambda_n \int_{(0,1)^d} fg_n \, dx,$$
and
$$C \inf \sum_{n=1}^\infty |\lambda_n| \leq \|\varphi\|_{B_\ast} \leq \inf \sum_{n=1}^\infty |\lambda_n|,$$
where $C > 0$ is an absolute constant, and the infimum is taken over all representations of $\varphi$.
\end{theorem}
\begin{remark}
As expected, the result shows that $\bbb_\ast$ is continuously contained in the atomic Hardy space $H^1((0,1)^d)$. In particular, the convergence of the series \eqref{eq:Batdecomp} can be understood in $H^1$, and thus in $L^1$.
\end{remark}
As mentioned before, we will also show the following biduality and distance result. The meaning of the duality relations will be made more precise in Section~\ref{sec:bidual}. Let $\UC \subset \bbb_0$ denote the space of uniformly continuous functions on $(0,1)^d$ (modulo constants). Note that $\UC$ is dense in $\bbb_0$, by the remark after Lemma~\ref{ap-property}.
\begin{theorem} \label{thm:Bbidual}
We have that $\bbb_0^\ast = \bbb_\ast$ and $\bbb_0^{\ast \ast} = \bbb_\ast^\ast = \bbb$, isometrically via the $L^2((0,1)^d)$-pairing. For any $f \in \bbb$ it holds that
\begin{equation} \label{eq:Bdistformula}
	  \distanza(f,  \bbb_0) = \distanza(f,  \UC) = \min_{g \in \mathcal{B}_0} \|f - g\|_{\mathcal{B}} = \varlimsup_{\varepsilon\rightarrow 0}\;[ f]_\varepsilon.
\end{equation}
\end{theorem}
\begin{remark}
The result of \cite{Per17} implies that $\bbb_0$ is an $M$-ideal in $\bbb$. Thus, by $M$-structure theory \cite{HSW75}, the minimizer of \eqref{eq:Bdistformula} exists but is never unique, unless $f \in \bbb_0$.
\end{remark}
\begin{remark}
If the constraint on the cardinality, $|\fff_\eee|\le \eee^{1-d}$, is removed from the definition of $\bbb$, we find the familiar space $\BV$ of functions of bounded variation on $(0,1)^d$. Indeed, it can easily be shown that $f\in \BV$ if and only if
\[\sup_{0<\eee\leq 1}\;\eee^{d-1}\sup_{\mathcal G_\eee}\sum_{Q_\eee\in \mathcal G_\eee}\frac{1}{|Q_{\varepsilon}|}\int_{Q_{\varepsilon}}\left|f(x)-f_{Q_\varepsilon}\right|\,dx< \infty,\]
where the supremum now runs over all families $\mathcal G_\eee$ of pairwise disjoint $\eee$-cubes contained in $(0,1)^d$. Compare with \cite{FFGS, FMS1}. Moreover, the above quantity is equivalent to the total variation $|Df|((0,1)^d)$. In this case, the corresponding ``vanishing subspace'' is trivial:
\[\lim_{\eee \to 0}\;\eee^{d-1}\sup_{\mathcal G_\eee}\sum_{Q_\eee\in \mathcal G_\eee}\frac{1}{|Q_{\varepsilon}|}\int_{Q_{\varepsilon}}\left|f(x)-f_{Q_\varepsilon}\right|\,dx=0\]
if and only if $f$ is constant.
\end{remark}

\section{Atomic decompositions} \label{sec:atomdecomp}
We will suppose that $X$ is reflexive, while letting $Y$ be any Banach space. In this section, let $(L_n)_{n=1}^\infty$ be a given sequence of bounded operators $L_n \colon X \to Y$, and define
$$E = \{x \in X \, : \, \sup_n \|L_nx\|_Y < \infty\}.$$
We suppose that this defines a Banach space $E$ under the norm
$$\|x\|_E = \sup_n \|L_nx\|_Y,$$
and that $E$ is continuously contained in $X$. We will not attempt to give a general condition under which these two hypotheses hold. We also suppose that $E$ is dense in $X$ in the $X$-norm, for otherwise we may replace $X$ by the closure of $E$ in $X$.

We thus have an isometric embedding
$$V \colon E \to \ell^\infty(Y), \quad Vx(n) = L_nx.$$
The dual of $E$ can therefore be represented as
$$E^\ast \simeq \ba(\mathbb{N}, Y^\ast)/VE^{\perp},$$ %= \left(\ell^1(Y^\ast) \oplus_1 \ell^1(Y^\ast)^{\perp} \right)/VB^{\perp} $$
where $\ba(\mathbb{N}, Y^\ast)$ denotes the space of finitely additive $Y^\ast$-valued set functions on $\mathbb{N}$ of bounded variation \cite{Leon76}, and $VE^\perp$ is its subspace of elements annihilating $VE$. Note that $\ell^1(Y^\ast)$ is naturally understood as the subspace of $\ba(\mathbb{N}, Y^\ast)$ consisting of the countably additive measures.
\begin{theorem} \label{thm:atomicdecomp}
$E$ has an isometric predual $E_\ast$,
$$E_\ast = \ell^1(Y^\ast)/P,$$
where $P = VE^{\perp} \cap \ell^1(Y^\ast)$.
That is, every $x \in E$ corresponds to a functional on $\ell^1(Y^\ast)/P$ given by
\begin{equation} \label{eq:dualpairing}
x ((y_n^\ast)) = \sum_{n=1}^\infty y_n^{\ast}(L_n x) = \sum_{n=1}^\infty L_n^\ast y_n^{\ast}(x), \quad (y_n^\ast) \in \ell^1(Y^\ast),
\end{equation}
and conversely every bounded functional on $\ell^1(Y^\ast)/P$ is given by a unique $x \in E$ according to \eqref{eq:dualpairing}. The norm of $x$ as a functional is equal to its norm as an element of $E$.
\end{theorem}
\begin{proof}
Through the canonical embedding $\iota \colon Y \to Y^{\ast \ast}$, we may instead consider the embedding
$$W \colon E \to \ell^{\infty}(Y^{\ast \ast}), \quad Wx(n) = \iota (Vx(n)).$$
Note that $\ell^{\infty}(Y^{\ast \ast}) = \ell^{1}(Y^\ast)^\ast$ \cite{Leon76}. To see that $E$ is a dual space we simply have to verify that $WE$ is weak-star closed in $\ell^\infty(Y^{\ast \ast})$. By the Krein-Smulian theorem, it is enough to check that $WE \cap D_1$ is weak-star closed, where $D_1$ is the closed unit ball with centre $0$ in $\ell^\infty(Y^{\ast \ast})$.

Suppose therefore that $(x_\alpha)$ is a net in the unit ball of $E$ such that $(L_n x_\alpha) \to (y^{\ast \ast}_n)$ weak-star,  where $(y^{\ast \ast}_n) \in \ell^\infty(Y^{\ast \ast})$. Since $E$ is continuously contained in $X$ and $X$ is reflexive, there is a subnet $(x_{\alpha'})$ such that $x_{\alpha'} \to x$ weakly in $X$, for some $x \in X$. For every fixed $n$, $L_n \colon X \to Y$ is continuous, and thus $L_n x_{\alpha'} \to L_nx$ weakly in $Y$. Since closed balls of $Y$ are weakly closed, we conclude that $x$ belongs to the unit ball of $E$. Of course, for every fixed $n$ we also know by hypothesis that $L_n x_{\alpha'} \to y_n^{\ast \ast}$ weak-star in $Y^{\ast \ast}$. Then, for every $y^{\ast} \in Y^\ast$,
$$y_n^{\ast \ast}(y^\ast) = \lim_{\alpha'} y^{\ast}(L_n x_{\alpha'}) = y^{\ast} (L_n x).$$
Hence $(y_n^{\ast \ast}) = (L_n x)$, demonstrating that $WE \cap D_1$ is weak-star closed.

We have shown that $E \simeq WE$ is a dual space, with predual given by
\begin{equation*}
(WE)_* = \ell^1(Y^\ast)/^{\perp}WE = \ell^1(Y^\ast)/P. \qedhere
\end{equation*}
\end{proof}
\begin{remark}
	If $Y$ is a dual space with predual $Y_\ast$, a similar argument shows that
	$$E_\ast = \ell^1(Y_\ast)/^{\perp}VE.$$
\end{remark}
To understand Theorem~\ref{thm:atomicdecomp}, note that finite sums $e_\ast = \sum L_n^\ast y^\ast_n$ belong both to $E_\ast$ and $X_\ast = X^\ast$, and the action of elements $x \in E$ on $e_\ast$ is identical,
$$x((y_n^{\ast})) = e_\ast(x).$$
Clearly, such finite sums are dense in $E_\ast$.
In fact, by restricting the action of $x^\ast \in X^\ast$ from $X$ to $E$, all of $X^\ast$ is continuously contained (and thus dense) in $E_\ast$.

\begin{theorem} \label{thm:xstar}
$X^\ast$ is continuously contained in $E_\ast$. Thus, for any $x^\ast \in X^\ast$ there is a sequence $(y_n^\ast) \in \ell^1(Y^\ast)$ such that $\|(y_n^\ast)\|_{\ell^1(Y^\ast)} \leq C\|x^\ast\|_{X^\ast}$ and
$$x^\ast(x) = \sum_{n=1}^\infty L_n^\ast y_n^\ast(x), \quad x \in E.$$
Here $C > 0$ is an absolute constant.
In particular, if $X$ is separable, then $E_\ast$ is separable.
\end{theorem}
\begin{proof}
Note first that it is clear that every $x^\ast \in X^\ast$ induces an element $x^\ast \in E^\ast$, since $E$ is continuously contained in $X$, and the implied map is injective.  To see that in fact $x^\ast \in E_\ast$, we need to verify that $x^\ast$ is weak-star continuous on $E$.  That is, we need to verify that $E \cap \ker x^\ast$ is weak-star closed in $E$. By the Krein-Smulian theorem, it is sufficient to consider $D_1 \cap E \cap \ker x^\ast$, where $D_1$ is the closed unit ball with centre $0$ in $E$. Let $(x_\alpha)$ be a net in this intersection such that $x_\alpha \to x$ weak-star in $E$. Since $E$ is continuously contained in $X$, there is a subnet $(x_{\alpha'})$ such that $x_{\alpha'} \to x_0$ weakly in $X$ for some $x_0 \in X$. By the same argument as in the proof of Theorem~\ref{thm:atomicdecomp}, we then have that $x_0 \in D_1 \cap E \cap \ker x^\ast$. Furthermore, for every $n$ and $y^\ast \in Y^\ast$, we find that
$$y^\ast(L_n x_0) = \lim_{\alpha'}  y^\ast(L_n x_{\alpha'}) =  y^\ast(L_n x).$$
Hence $L_n x_0 = L_n x$ for every $n$, and thus $x_0 = x$ (by the assumption that the norm of $E$ is indeed a norm).
\end{proof}
\section{Biduality and distance formulas} \label{sec:bidual}
In this section we will recall the framework and results of \cite{Per13, Per17}, imposing additional structure on $E$. We now suppose that $X$ is separable, in addition to being reflexive. $Y$ is still an arbitrary Banach space. We assume that $E$ is normed by a collection $\mathcal{L}$ of operators $L \colon X \to Y$, equipped with a topology that is $\sigma$-compact, locally compact, Hausdorff, and such that $\mathcal{L} \ni L \to Lx \in Y$ is continuous for every fixed $x \in X$. To reconcile this framework with that of Section~\ref{sec:atomdecomp}, we additionally assume that the topology is separable.

Let
$$E = \{x \in X \, : \, \sup_{L \in \mathcal{L}} \|Lx\|_Y < \infty \},$$
and
$$E_0 = \{x \in E \, : \, \varlimsup_{\mathcal{L} \ni L \to \infty} \|Lx\|_Y = 0\}.$$
Here $L\to \infty$ is understood in the usual sense of escaping all compacts. As before, we assume that $E$ is a Banach space under the norm $\|x\|_E  = \sup_{L \in \mathcal{L}} \|Lx\|_Y$, and that $E$ is continuously contained and dense in $X$.

Additionally, we have to assume an approximation property.
\bigskip

\noindent (AP) For every $x \in E$ there is a sequence $(x_n) \subset E_0$ such that $\|x_n\|_E \leq \|x\|_E$ and $x_n \to x$ in $X$.
\bigskip

\noindent Since closed balls of $E_0$ may be viewed as bounded convex subsets of $X$, an equivalent reformulation of (AP) is obtained by replacing the strong convergence of $x_n \to x$ by weak convergence in $X$.

We shall not recall the notion of an $M$-ideal here, but instead refer the interested reader to \cite{HWW93}. Let $\iota \colon E_0 \to E_0^{\ast \ast}$ denote the canonical embedding.
\begin{theorem}[\cite{Per13, Per17}] \label{thm:bidual1}
Suppose that {\rm (AP)} holds. Then
\begin{itemize}
\item $E_0^{\ast \ast} \simeq E$ isometrically via the $X$-$X^\ast$-pairing. More precisely, let $I \colon E_0 \to X$ denote the inclusion operator, and let $U = I^{\ast \ast}$. Then $U(E_0^{\ast \ast}) = E$ and, considered as an operator $U \colon E_0^{\ast \ast} \to E$, $U$ is the unique isometric isomorphism such that $U \iota x = x$ for all $x \in E_0$.
\item $E_0$ is an $M$-ideal in $E$.
\item For every $x \in E$ it holds that
$$\min_{x_0 \in E_0} \|x - x_0\|_{E} = \varlimsup_{L \to \infty} \|Lx\|_Y.$$
\end{itemize}
\end{theorem}
The uniqueness of $U$ has not been explicitly stated before, but follows as in \cite[Theorem~2]{Per19}.

Choosing a dense sequence of operators $(L_n)$ in $\mathcal{L}$ , we have by continuity that
$$\|x\|_E = \sup_n \|L_n x\|_Y,$$
allowing us to apply the results of Section~\ref{sec:atomdecomp}. Theorems~\ref{thm:atomicdecomp} and \ref{thm:xstar} yield an isometric isomorphism $J \colon E \to E_\ast^\ast$ such that $Jx(x^\ast) = x^\ast(x)$ for $x^\ast \in X^\ast \subset E_\ast$ and $x \in E$. Since $E_0$ is an $M$-ideal in $E_0^{\ast \ast}$, $E_0^\ast$ is a strongly unique predual, implying that the isometry $JU \colon E_0^{\ast \ast} \to E_{\ast}^\ast$ must arise as the adjoint $JU = K^\ast$ of an isometric isomorphism $K \colon E_\ast \to E_0^{\ast}$. Note, for $x^\ast  \in X^\ast \subset E_\ast$ and $x \in E_0$, that
$$Kx^\ast(x) = JU\iota x(x^\ast) = x^\ast(x).$$
\begin{theorem} \label{thm:bidual2}
Suppose that {\rm (AP)} holds. Then the identity map, $K(x^\ast) = x^\ast$,  $x^\ast \in X^\ast$, extends to an isometric isomorphism $K \colon E_\ast \to E_0^\ast$.
\end{theorem}
Less formally, we might simply say that $E_0^\ast = E_\ast$ and that $E_\ast^\ast = E$, isometrically via the $X$-$X^\ast$-pairing.

\section{The Bourgain--Brezis--Mironescu space}
We restrict the discussion to $d \geq 2$. If $d=1$, then $\bbb = \BMO$ and the results are already known. If $d \geq 2$, \cite[Theorem~2]{BBM15} states that $\bbb$ is continuously contained in the Marcinkiewicz space $L^{d/(d-1), \infty}$ (modulo constants). Therefore, as our choice of separable reflexive space containing $\bbb$, we may take $X = L^p((0,1)^d)/\mathbb{C}$, where $1 < p < d/(d-1)$.  Let $Y = L^1((0,1)^d)$.

As before, let $\mathcal{F}_\varepsilon$ denote a collection of disjoint open $\varepsilon$-cubes such that $|\mathcal{F}_\varepsilon| \leq \varepsilon^{1-d}$. Thus $\mathcal{F}_\varepsilon = (Q_\varepsilon(a_j))_{j=1}^m$, where $Q_\varepsilon(a_j) \subset (0,1)^d$ is an $\varepsilon$-cube centered at $a_j \in (0,1)^d$, and $m \leq  \varepsilon^{1-d}$. To each such collection $\mathcal{F}_\varepsilon$, associate the operator $L_{\mathcal{F}_\varepsilon} \colon L^p/\mathbb{C} \to L^1$ given by
	$$L_{\mathcal{F}_\varepsilon} f = \varepsilon^{d-1}\frac{1}{|Q_\varepsilon|}\sum_{j=1}^{|\mathcal{F}_{\varepsilon}|}\chi_{Q_\varepsilon(a_j)}(f-f_{Q_\varepsilon(a_j)}).$$
Since the cubes are disjoint, we then have that
	$$\|L_{\mathcal{F}_\varepsilon} f\|_{L^1} = \varepsilon^{d-1} \frac{1}{|Q_\varepsilon|} \sum_{j=1}^{|\mathcal{F}_{\varepsilon}|} \int_{Q_\varepsilon(a_j)}|f - f_{Q_\varepsilon(a_j)}| \, dx,$$
and thus that
$$\bbb = \{f \in L^p/\mathbb{C} \, : \, \sup_{0 < \varepsilon \leq 1} \sup_{\mathcal{F}_\varepsilon} \|L_{\mathcal{F}_\varepsilon} f\|_{L^1} < \infty\}.$$

To place ourselves in the framework of Section~\ref{sec:bidual}, we have to construct an appropriate topology on the set of collections $\mathcal{F}_\varepsilon$. Each collection is uniquely determined by $\varepsilon$ and the centres $a_1, \ldots, a_m$ of the cubes $(0,1)^d \supset Q_\varepsilon(a_j) \in \mathcal{F}_\varepsilon$. We thus write $\mathcal{F}_\varepsilon = \mathcal{F}(\varepsilon, a_1, \ldots, a_m)$. For each $m \in \mathbb{N}$, let
$$\mathcal{L}_m = \{\mathcal{F}_\varepsilon = \mathcal{F}(\varepsilon, a_1, \ldots, a_m) \, : \, |\mathcal{F}_\varepsilon| = m\},$$
with the topology induced by the parametrization $(\varepsilon, a_1, \ldots, a_m) \in (0, 1] \times (0,1)^{md}$. Consider the map $\pi_m \colon \mathcal{L}_m \to (0, m^{-1/(d-1)}]$ given by $\pi_m(\mathcal{F}_\varepsilon) = \varepsilon$. Since the cubes $Q_\varepsilon(a_j)$ are open, the map $\pi_m$ is proper. That is, $\pi^{-1}_m([\delta, m^{-1/(d-1)}])$ is compact for every $0 < \delta \leq m^{-1/(d-1)}$.  By the dominated convergence theorem, the map $\mathcal{L}_m \ni \mathcal{F}_\varepsilon \mapsto L_{\mathcal{F}_\varepsilon}f \in L^1$ is continuous for every fixed $f \in L^p$ and $m$.

We endow the full collection $\mathcal{L} = \{\mathcal{F_\varepsilon}\}$ with the disjoint union topology,
$$\mathcal{L} = \coprod_{m=1}^\infty \mathcal{L}_m.$$
By the corresponding properties of $\mathcal{L}_m$, $\mathcal{L}$ is $\sigma$-compact, locally compact, Hausdorff, separable, and $\mathcal{L} \ni \mathcal{F}_\varepsilon \mapsto L_{\mathcal{F}_\varepsilon}f \in L^1$ is continuous for every $f \in L^p$. Furthermore, the map
$$\pi = \coprod_{m=1}^\infty \pi_m \colon \mathcal{L} \to (0,1], \quad \pi(\mathcal{F}_\varepsilon) = \varepsilon,$$
is proper, since $\pi_m$ is proper, and for every $\delta > 0$ there is an $N$ such that $\pi^{-1}([\delta,1]) \subset \coprod_{m \leq N} \mathcal{L}_m$. For a continuous function $F \colon \mathcal{L} \to \mathbb{R}$, this implies that
$$\varlimsup_{\mathcal{L} \ni \mathcal{F}_\varepsilon \to \infty} F(\mathcal{F}_\varepsilon) =\varlimsup_{\varepsilon\to0} \sup_{\mathcal{F}_\varepsilon} F(\mathcal{F}_\varepsilon).$$

Hence the vanishing Bourgain--Brezis--Mironescu space $\bbb_0$ fits into the framework,
$$\bbb_0 = \{f \in \bbb \, : \, \varlimsup_{\varepsilon\to0} \sup_{\mathcal{F}_\varepsilon} \|L_{\mathcal{F}_\varepsilon}f \|_{L^1} = 0\}.$$
Before applying our results, we must prove that (AP) holds.
\begin{lemma}\label{ap-property}
For every $f\in \bbb$, there is a sequence $(f_n) \subset \bbb_0$ such that $\|f_n\|_{\bbb_0} \leq \|f\|_{\bbb}$ and $f_n \to f$ in $L^p$, $1 < p < d/(d-1)$.
\end{lemma}
\begin{proof}
For $0<\ttt<\frac12$, set $\tvet=(\ttt,\ldots,\ttt)\in \reale^d$, and let
\[g(x)=f((1-2\ttt)x+\tvet),\quad x\in (0,1)^d\,.\]
Note that $g(x)$ is actually defined for all points $x \in (-\ttt, 1+\ttt)^d$.
Choose a function $\psi \in C_c(\reale^d)$ such that $\psi \ge 0$, $\supp \psi \subset (-1,1)^d$, and $\int \psi (y)\,\de y=1$. Fix $\ttt$ and let $\vf(y) = \vf_\ttt(y) = \ttt^{-d}\psi(y/\ttt)$. Throughout the proof, integration with respect to $y$ will be understood to be taken over $\supp\vf \subset (-\ttt, \ttt)^d$. Let
\[h(x)=h_\ttt(x)=\vf*g(x)=\int \vf(y)\,g(x-y)\,\de y\,.\]
For cubes $Q_\varepsilon \subset (0,1)^d$ , we use the notation
\[M(f,Q_\varepsilon)=\frac{1}{|Q_\varepsilon|} \int_{Q_\varepsilon} |f-f_{Q_\varepsilon}|\, \de x.\]
Note that
\begin{align*}
 h(x)-h_{Q_\varepsilon} &=h (x)-\frac1{|Q_\varepsilon|}\int_{Q_\varepsilon}\left(\int\vf(y)\,g(z-y)\,\de y\right)\de z \\
&=\int\vf(y)\left(g(x-y)- g_{Q_\varepsilon - y} \right) \, \de y,
\end{align*}
and thus
\begin{equation} \label{eq:APest}
M(h, Q_\varepsilon)\leq \int\vf(y)\,M(g,Q_\varepsilon-y)\,\de y.
\end{equation}
Furthermore,
\begin{align*}
M(g,Q_\varepsilon-y) &= \frac{1}{|Q_\varepsilon(y,\ttt)|}\int_{Q_\varepsilon(y,\ttt)}\left|f(u)-\frac{1}{|Q_\varepsilon(y,\ttt)|}\int_{Q_\varepsilon(y,\ttt)}f(v)\,\de v\right|\de u\\
&=M(f,Q_\varepsilon(y,\ttt))
\end{align*}
where $Q_\varepsilon(y,\ttt)=\{(1-2\ttt)x+\tvet:x\in Q_\varepsilon-y\}$. Note that $Q_\varepsilon(y,\ttt) \subset (0,1)^d$ is an $(1-2\ttt)\varepsilon$-cube, for every $y$.

Now consider a collection $\mathcal{F}_\varepsilon$ of disjoint $\varepsilon$-cubes such that $|\fff_\eee| \leq \eee^{1-d}$. For each $y\in\supp \vf$,
$$\mathcal{F}_\varepsilon(y, \ttt) = \{Q_\varepsilon(y,\ttt) \, : \, Q_\varepsilon \in \mathcal{F}_\varepsilon \}$$
forms a collection of disjoint $(1-2\ttt)\eee$-cubes, with
$$|\mathcal{F}_\varepsilon(y, \ttt)|=|\fff_\eee| \le ((1-2\ttt)\,\eee)^{1-d}.$$ Hence by \eqref{eq:APest},
\begin{equation}\label{eq:confronto}
\begin{array}{rl}
\ds \eee^{d-1}\sum_{Q_\varepsilon\in\fff_\eee}M(h,Q_\varepsilon)\kern-.7em &\ds\leq \frac1{(1-2\ttt)^{d-1}} \int \varphi(y)\,((1-2\ttt)\eee)^{d-1}\sum_{\mathcal{F}_\varepsilon(y, \ttt)}M(f,Q_\varepsilon(y,\ttt)) \de y\\
&\ds\leq \frac1{(1-2\ttt)^{d-1}}\,\|f\|_{\bbb}.
\end{array}
\end{equation}
Therefore
\[\|h\|_{\bbb}\le \frac1{(1-2\ttt)^{d-1}}\,\|f\|_{\bbb}\,.\]
To conclude, choose a sequence $(\ttt_n)$ such that $\ttt_n \to 0$, and let
$$f_n = (1-2\ttt_n)^{d-1} h_{\ttt_n}.$$
Then $(f_n) \subset \bbb_0$, $\|f_n\|_{\bbb_0} \leq \|f\|_{\bbb}$, and $f_n \to f$ in $L^p$.
\end{proof}
\begin{remark}
Note that if $f\in \bbb_0$, then the sequence $(f_n)$ constructed in the proof of Lemma~\ref{ap-property} converges to $f$ in $\bbb$. Hence the space $\UC$ of uniformly continuous functions on $(0,1)^d$ is dense in $\bbb_0$.

Indeed, the inequality \eqref{eq:confronto} shows that
\[[f_n]_\eee\leq [f]_{(1-2\ttt)\eee}, \quad 0 < \varepsilon \leq 1.\]
Since $f\in \bbb_0$, there is for every $\sigma>0$ an $\eee_\sigma \in (0,1)$ such that for $0<\eee<\eee_\sigma$ we have that
\[[f]_\eee<\sigma.\]
Thus, for such $\eee$,
\[[f-f_n]_\eee\leq [f]_\eee+[f_n]_\eee<2\sigma.\]
On the other hand, for any family $\mathcal{F}_\eee$ of disjoint $\eee$-cubes we have the trivial estimate
\[\eee^{d-1}\sum_{Q_e\in\mathcal{F}_\eee}M(f-f_n,Q_\eee)\leq 2\eee^{-1}\|f-f_n\|_{L^1((0,1)^d)}.\]
Hence
\[\lim_n\sup_{\eee_\sigma\leq \eee<1}[f-f_n]_\eee=0,\]
since $f_n\to f$ in $L^1((0,1)^d)$. It follows that $f_n \to f$ in $\bbb$, since $\sigma$ was arbitrary.
\end{remark}

To prove Theorem~\ref{thm:Bpredual}, choose a dense sequence $(\mathcal{F}_{\varepsilon_n}^n) \subset \mathcal{L}$. Then Theorem~\ref{thm:atomicdecomp} yields that for every $\varphi \in \bbb_\ast$ there is a sequence $(y_n^\ast) \in \ell^1(L^\infty( (0,1)^d))$ of comparable norm and such that
$$\varphi = \sum_{n=1}^\infty L_{\mathcal{F}_{\varepsilon_n}^n}^\ast y_n^\ast = \sum_{n=1}^\infty L_{\mathcal{F}_{\varepsilon_n}^n} y_n^\ast.$$
Then $\lambda_n = 2\|y_n^\ast\|_{L^\infty}$ and $g_n = L_{\mathcal{F}_{\varepsilon_n}^n} y_n^\ast/(2\|y_n^\ast\|_{L^\infty})$ satisfy the desired properties, $\varphi = \sum_n \lambda_n g_n$, and
$$\sum_{n=1}^\infty |\lambda_n| \leq C\|\varphi\|_{\bbb_\ast},$$
where $C$ is an absolute constant. Conversely, suppose that we are given a functional $\varphi = \sum_n \lambda_n g_n$ of the form in the theorem. Let $\mathcal{F}^n_{\varepsilon_n}$ denote the collection of cubes associated with $g_n$. If necessary we may complete $(\mathcal{F}^n_{\varepsilon_n})$ to a dense sequence. Theorem~\ref{thm:atomicdecomp} then shows that $\varphi$ is weak-star continuous on $\bbb$, and it is immediate that
$$\|\varphi\|_{\bbb_\ast} \leq \sum_{n=1}^\infty |\lambda_n|.$$

Theorem~\ref{thm:Bbidual} is an immediate consequence of Theorems~\ref{thm:bidual1} and \ref{thm:bidual2}.

\end{document}